\pdfoutput=1
\RequirePackage{ifpdf}
\ifpdf % We~are running pdfTeX in pdf mode
\documentclass[pdftex]{sigma}
\else
\documentclass{sigma}
\fi

\numberwithin{equation}{section}

\newtheorem{Theorem}{Theorem}[section]
\newtheorem{Corollary}[Theorem]{Corollary}
\newtheorem{Lemma}[Theorem]{Lemma}
\newtheorem{Proposition}[Theorem]{Proposition}
 { \theoremstyle{definition}
\newtheorem{Definition}[Theorem]{Definition}

\newtheorem{Example}[Theorem]{Example}
\newtheorem{Remark}[Theorem]{Remark} }

\DeclareMathOperator{\Id}{Id}
\DeclareMathOperator{\Span}{Span}
\DeclareMathOperator{\Ann}{Ann}
\DeclareMathOperator{\End}{End}
\DeclareMathOperator{\Hom}{Hom}
\DeclareMathOperator{\Aut}{Aut}

\usepackage{enumitem}
\usepackage{mathrsfs}
\usepackage{romanbar}
\newcommand{\II}{\mathrm{I\hspace{-.1em}I}}

\begin{document}
\allowdisplaybreaks

\newcommand{\arXivNumber}{2403.15968}

\renewcommand{\PaperNumber}{001}

\FirstPageHeading

\ShortArticleName{Symplectic Differential Reduction Algebras and Generalized Weyl Algebras}

\ArticleName{Symplectic Differential Reduction Algebras\\ and Generalized Weyl Algebras}

\Author{Jonas T.~HARTWIG~$^{\rm a}$ and Dwight Anderson WILLIAMS~{II}~$^{\rm b}$}

\AuthorNameForHeading{J.T.~Hartwig and D.A.~Williams~{II}}

\Address{$^{\rm a)}$~Department of Mathematics, Iowa State University, Ames, IA 50011, USA}
\EmailD{\href{mailto:jth@iastate.edu}{jth@iastate.edu}}
\URLaddressD{\url{http://jthartwig.net}}

\Address{$^{\rm b)}$~Department of Mathematics, Morgan State University, Baltimore, MD 21251, USA}
\EmailD{\href{dwight@mathdwight.com}{dwight@mathdwight.com}}
\URLaddressD{\url{https://mathdwight.com}}

\ArticleDates{Received March 27, 2024, in final form December 23, 2024; Published online January 01, 2025}

\Abstract{Given a map $\Xi\colon U(\mathfrak{g})\rightarrow A$ of associative algebras, with $U(\mathfrak{g})$ the universal enveloping algebra of a (complex) finite-dimensional reductive Lie algebra $\mathfrak{g}$, the restriction functor from $A$-modules to $U(\mathfrak{g})$-modules is intimately tied to the representation theory of an $A$-subquotient known as the \emph{reduction algebra} with respect to $(A,\mathfrak{g},\Xi)$. Herlemont and Ogievetsky described differential reduction algebras for the general linear Lie algebra $\mathfrak{gl}(n)$ as algebras of deformed differential operators. Their map $\Xi$ is a realization of $\mathfrak{gl}(n)$ in the $N$-fold tensor product of the $n$-th Weyl algebra tensored with $U(\mathfrak{gl}(n))$. In this paper, we further the study of differential reduction algebras by finding a presentation in the case when $\mathfrak{g}$ is the symplectic Lie algebra of rank two and $\Xi$ is a canonical realization of $\mathfrak{g}$ inside the second Weyl algebra tensor the universal enveloping algebra of $\mathfrak{g}$, suitably localized. Furthermore, we prove that this differential reduction algebra is a generalized Weyl algebra (GWA), in the sense of Bavula, of a new type we term skew-affine. It is believed that symplectic differential reduction algebras are all skew-affine GWAs; then their irreducible weight modules could be obtained from standard GWA techniques.}


\Classification{16S15; 	16S32; 17B35; 17B37; 81R05}

\section{Introduction}\label{sec:Intro}

\subsection{Background}\label{subsec:Background}
Mickelsson addressed the problem of reducing representations of Lie algebras over a subalgebra with the introduction of step algebras \cite{mickelssonStepAlgebrasSemisimple1973}, which are now known as Mickelsson algebras. These algebras are part of a broader family, collectively called \emph{reduction algebras} \cite{khoroshkinDiagonalReductionAlgebras2010}.
Specific applications and developments relying on reduction algebras are found in the decomposition of tensor product modules (via actions of diagonal reduction algebras) \cite{hartwigDiagonalReductionAlgebra2022c,khoroshkinDiagonalReductionAlgebras2010}, higher-order Fischer decompositions in harmonic analysis \cite{debieHarmonicTransvectorAlgebra2017}, conformal field theory \cite{conleyBoundedLengthRepresentations2001,conleyConformalSymbolsAction2009}, and the construction of wave functions \cite{asherovaProjectionOperatorsSimple1973} in the realm of theoretical particle and nuclear physics, as examples. The last paper formalized extremal projectors, and Zhelobenko \cite{zhelobenkoExtremalProjectorsGeneralized1989} is credited with realizing the importance of extremal projectors in a scheme to determine generators and relations of localized reduction algebras. See \cite{zhelobenkoPrincipalStructuresMethods2006} for an overview of the construction of generalized Mickelsson algebras.

The crucial observation is that reduction algebras associated to a reductive Lie algebra $\mathfrak{g} = \mathfrak{g}_{-} \oplus \mathfrak{h} \oplus \mathfrak{g}_{+}$ generally arise from a triple of data, that is, a Lie algebra homomorphism $\Xi\colon \mathfrak{g} \rightarrow A$ from $\mathfrak{g}$ to an associative algebra $A$ or the resulting algebra map (perhaps also called $\Xi$) from the universal enveloping algebra $U(\mathfrak{g})$ of $\mathfrak{g}$ to $A$. We set $I$ to be the left ideal in $A$ generated by the positive nilpotent subalgebra $\mathfrak{g}_{+} \subset \mathfrak{g}$ via the $\Xi$-induced action of $\mathfrak{g}$ on $A$. Then the reduction algebra $Z = Z(A,\mathfrak{g},\Xi)$, up to a choice of localization, is the subquotient $N/I$, where the normalizer $N$ is the largest subalgebra of $A$ containing $I$ as a two-sided ideal. A common case occurs when $\Xi$ is an embedding of $\mathfrak{g}$ into the universal enveloping algebra $U(\mathfrak{t})$ of a Lie algebra $\mathfrak{t} \supset \mathfrak{g}$ containing $\mathfrak{g}$, and we write simply $Z = Z(\mathfrak{t},\mathfrak{g})$. The same abbreviation of the map occurs in the super analogue or whenever the Lie (super)algebra homomorphism is understood. There is also a $q$-analog of this construction using Drinfeld--Jimbo quantum groups.

We now emphasize two points. The first is that reduction algebras are associative algebras with their own lanes of study. There are a great number of results on reduction algebras with many summarized at the end of Tolstoy's commemorative paper \cite{tolstoy2005fortieth}. Khoroshkin and Ogievetsky furthered results by giving a complete presentation of the diagonal reduction algebra \cite{khoroshkinDiagonalReductionAlgebras2010} of type $A$ of any rank.
With motivation drawn from the second author's work \cite{williamsii2020BasesInfinitedimensionalRepresentations,williamsii2024actionosp12npolytensor} on finding explicit bases of tensor product representations of the orthosymplectic Lie superalgebras~$\mathfrak{osp}(1|2n)$, we gave a complete presentation \cite[Section~3.5]{hartwigDiagonalReductionAlgebra2022c} of the diagonal reduction superalgebra $Z(\mathfrak{osp}(1|2) \times \mathfrak{osp}(1|2), \mathfrak{osp}(1|2))$, the first recorded diagonal reduction algebra associated to a classical \cite{kacClassificationSimpleLie1979} Lie superalgebra. In the sequel \cite{hartwigGhostCenterRepresentations2023}, we classified finite-dimensional and certain infinite-dimensional representations of $Z(\mathfrak{osp}(1|2) \times \mathfrak{osp}(1|2), \mathfrak{osp}(1|2))$.

The second matter is that homomorphisms to the $n$-th Weyl algebra $A_{n}$ (see the use of \emph{canonical realizations} in \cite[Definition 1]{havlicekCanonicalRealizationsLie1975}) yield examples of reduction algebras important to quantum deformations and noncommutative geometry. A particular example is the ring $\mathcal{D}_{\mathbf{h}}(n,N)$ of $\mathbf{h}$-deformed differential operators \cite{herlemontDifferentialCalculusMathbf2017}. It is the reduction algebra defined from the diagonal map sending $\mathfrak{gl}(n)$ to (a localization of) the tensor product $U(\mathfrak{gl}(n)) \otimes A_{nN}$; the first factor relies on the definition of $U(\mathfrak{gl}(n))$, and the second factor uses $N$-copies of an oscillator realization of~$\mathfrak{gl}(n)$.

In this paper, we combine the two points of emphasis above: We initiate the study of the symplectic analogue of $\mathcal{D}_{\mathbf{h}}(2,1)$ by extending a classical oscillator realization of the rank two symplectic Lie algebra $\mathfrak{sp}(4)$ within $A_{2}$ such that the codomain algebra $A$ is the tensor product~${A_2\otimes U(\mathfrak{sp}(4))}$. We localize $A$ to an algebra $\mathscr{A}$ in order to satisfy certain conditions allowing for the use of the $\mathfrak{sp}(4)$ extremal projector.
After sorting out preliminaries in Section \ref{sec:PreliminariesRedAlg}, we use the extremal projector to present a complete set of generators and relations of the \emph{symplectic differential reduction algebra} $D(\mathfrak{sp}(4)) \cong Z(\mathscr{A},\mathfrak{sp}(4))$ in Section \ref{sec:sp4reduction}. Furthermore, we remark that there is a homomorphism from the diagonal reduction algebra of $\mathfrak{sp}(4)$ to $D(\mathfrak{sp}(4))$, and therefore our presentation of the latter should be helpful in finding a presentation of the former.
Lastly, we show in Section \ref{sec:GWA} that after a change of generators, $D(\mathfrak{sp}(4))$ is manifestly a generalized Weyl algebra (GWA) of a new and interesting type. In particular, the irreducible weight modules over $D(\mathfrak{sp}(4))$ can now be classified by understanding orbits and breaks in the maximal spectrum. See \cite{bavula1992generalized,hartwigNoncommutativeFiberProducts2018,mazorchukAssociativeAlgebrasRelated2003}.

\section{Preliminaries}\label{sec:PreliminariesRedAlg}
The base field is $\mathbb{C}$ except when specified otherwise.

\subsection{Oscillator realization and Chevalley generators}\label{subsec:Oscillator-Chevalley}

{\bf Weyl algebras.}
Let $A_{n}$ denote the $n$-th Weyl algebra. Thus, $A_{n}$ is the complex associative algebra with generators~$x_{i}$, $\partial_{i} = \frac{\partial}{\partial x_{i}}$, for $i=1,2,\dots,n$, and relations
\begin{align}\label{eq:weylrels}
	x_ix_j-x_ jx_i =0,\qquad
	\partial_i\partial_j-\partial_j\partial_i =0,\qquad
	\partial_ix_j-x_j\partial_i =\delta_{ij},
\end{align}
where $\delta_{ij}$ is the Kronecker delta. The Weyl algebra $A_{n}$ carries an involutive anti-automor\-phism\footnote{That is, $\vartheta$ is $\mathbb{C}$-linear, $\vartheta(ab)=\vartheta(b)\vartheta(a)$, and $\vartheta^2=\Id_{A_n}$.}~$\vartheta$, called the symplectic Fourier transform, uniquely determined by
\begin{equation}\label{eq:vartheta}
	\vartheta(x_{i})=\partial_{i}.
\end{equation}

{\bf The rank two symplectic Lie algebra.}
Let $\mathfrak{g}=\mathfrak{sp}(4)$ be the $10$-dimensional complex simple Lie algebra of Cartan type $C_2$. It consists of all $4\times 4$ matrices of the form $\left[\begin{smallmatrix}A&B\\C&D\end{smallmatrix}\right]$ where~$B$ and $C$ are symmetric $2\times 2$ matrices, $A$ is any $2\times 2$ matrix, and $D=-A^{\mathsf{T}}$.
The oscillator realization of $\mathfrak{sp}(4)$ is a canonical algebra homomorphism $\omega\colon U(\mathfrak{sp}(4))\to A_2$; see, for example, \cite[Section~4.6]{dixmierEnvelopingAlgebras1996} or \cite{havlicekCanonicalRealizationsLie1976}.
The image of $\mathfrak{sp}(4)$ is $\Span_\mathbb{C}\{uv+vu\mid u,v\in\{x_1,x_2,\partial_1,\partial_2\}\}$.
Since $\mathfrak{sp}(4)$ is simple, we may thus identify $\mathfrak{sp}(4)$ with the $10$-dimensional Lie subalgebra of $A_{2}$ spanned by $\{a_{ij}, b_{ij}, c_{ij} \mid i = 1,2\}$, where
\begin{equation}\label{eq:sp4-oscillator}
	a_{ij} = \frac{1}{2}(x_i\partial_j+\partial_jx_i)=x_i\partial_j+\frac{1}{2}\delta_{ij},\qquad
	b_{ij} = b_{ji} = x_ix_j,\qquad
	c_{ij} = c_{ji} = \partial_i\partial_j.
\end{equation}

Note that $\vartheta$ from \eqref{eq:vartheta} preserves the subspace $\mathfrak{sp}(4)\subset A_2$. We identify $\vartheta|_{\mathfrak{sp}(4)}$ with the involutive Lie algebra anti-automorphism $\tau$ given by
$\tau(a_{ij})=a_{ji}$, $ \tau(b_{ij})=c_{ij}$.

After this identification, the subspace $\mathfrak{h}= \mathbb{C} a_{11} \oplus \mathbb{C} a_{22}$ is a Cartan subalgebra of $\mathfrak{sp}(4)$. We use standard notation for the $8$ roots of $\mathfrak{sp}(4)$, which are
$\pm 2\epsilon_{1}$, $\thinspace \pm 2\epsilon_{2}$, $\thinspace \pm(\epsilon_{1}-\epsilon_{2})$, $\thinspace \pm(\epsilon_{1}+\epsilon_{2})$,
denoting the simple roots by $\alpha = \epsilon_1-\epsilon_2$, $\beta = 2\epsilon_2$. The set of positive roots is $\Phi_+=\{\beta, \beta+\alpha, \beta+2\alpha, \alpha\}$.
There are exactly two convex orderings \cite{itoClassificationConvexOrders2001} of the positive roots
\[%\label{eq:sp4-convex-orderings}
	\beta<\beta+\alpha<\beta+2\alpha<\alpha,\qquad
	\alpha<\beta+2\alpha<\beta+\alpha<\beta.
\]
The corresponding root system is shown in Figure \ref{fig:C2}.

\begin{figure}[ht]
	\centering
	\includegraphics{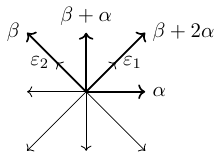}
	\caption{Root system of type $C_2$. Thick lines indicate a choice of positive roots.}
\label{fig:C2}
\end{figure}

Fix positive root vectors, with
${\rm i}=\sqrt{-1}$:
\begin{align*}
	e_{\alpha} = a_{12}, \qquad e_{\beta} = \frac{{\rm i}}{2}b_{22},\qquad
	e_{\beta+\alpha} =[e_{\alpha},e_\beta],\qquad e_{\beta+2\alpha}=\frac{1}{2}[e_\alpha,e_{\beta+\alpha}].
\end{align*}
These provide $\mathfrak{sl}(2)$-triples $(e_{\gamma},f_{\gamma},h'_{\gamma})$, where $f_{\gamma}=\tau(e_{\gamma})$ and $h'_{\gamma}=[e_{\gamma},f_{\gamma}]$:
\begin{gather*}
		\alpha\colon\ (a_{12},\thinspace a_{21},\thinspace a_{11} - a_{22}) = (x_{1}\partial_{2},\thinspace x_{2}\partial_{1},\thinspace x_{1}\partial_{1}-x_{2}\partial_{2}),\\
		\beta\colon\ \left(\frac{\rm i}{2}b_{22},\thinspace \frac{\rm i}{2}c_{22},\thinspace a_{22}\right) = \left(\frac{\rm i}{2}x^{2}_{2}x_{2},\thinspace \frac{\rm i}{2}\partial^{2}_{2},\thinspace x_{2}\partial_{2} + \frac{1}{2}\right),\\
		\beta + \alpha\colon\ ({\rm i} b_{12},\thinspace {\rm i} c_{12},\thinspace a_{11} + a_{22}) = ({\rm i}x_{1}x_{2},\thinspace {\rm i}\partial_{1}\partial_{2},\thinspace x_{1}\partial_{1} + x_{2}\partial_{2} + 1),\\
		\beta + 2\alpha\colon\ \left(\frac{\rm i}{2}b_{11},\thinspace \frac{\rm i}{2}c_{11},\thinspace a_{11}\right) = \left(\frac{\rm i}{2}x_{1}x_{1},\thinspace \frac{\rm i}{2}\partial_{1}\partial_{1},\thinspace x_{1}\partial_{1} + \frac{1}{2}\right).
\end{gather*}
We reserve $h_\gamma$ for the following integer-shifted coroots,
which will simplify computations
\[%\label{eq:shifted-coroots-lowercase}
	h_\alpha = h_\alpha',\qquad h_{\beta+2\alpha}=h_{\beta+2\alpha}'+1,\qquad h_{\beta+\alpha}=h_{\beta+\alpha}'+2,\qquad h_\beta=h_\beta',
\]
where $h'_{\beta+\alpha} = h_{\alpha} + 2h_{\beta}$ and $h'_{\beta+2\alpha} = h_{\alpha} + h_{\beta}$.
The simple coroots satisfy $\alpha(h_{\beta})=-1$, $\beta(h_{\alpha})=-2$.
Also, note that any scalar shift of $h'_{\alpha+\beta}$ yields the total degree derivation on $A_{2}$ in the adjoint representation:
$[x_{1}\partial_{1} + x_{2}\partial_{2} + 1, x_i]=x_i$ and $[x_{1}\partial_{1} + x_{2}\partial_{2} + 1, \partial_i]=-\partial_i$.

{\bf The map $\boldsymbol{\zeta}$.}
Composing the usual comultiplication in $U(\mathfrak{sp}(4))$ with the oscillator representation in the first tensor leg, we obtain an algebra homomorphism
\begin{equation}\label{eq:Homomorphism-sp4toA}
	\zeta\colon\ U(\mathfrak{sp}(4)) \rightarrow A_{2}\otimes U(\mathfrak{sp}(4)),\qquad v\mapsto \omega(v)\otimes 1+1\otimes v.
\end{equation}

We denote the images of the $\mathfrak{sl}(2)$-triples under $\zeta$ by $(E_\gamma,F_\gamma,H'_\gamma)$.
Similarly, put
\[H_\alpha = H_\alpha',\qquad H_{\beta+2\alpha}=H_{\beta+2\alpha}'+1,\qquad H_{\beta+\alpha}=H_{\beta+\alpha}'+2,\qquad H_\beta=H_\beta'.\]

\begin{Remark}
	The tensoring of $A_{2}$ with $U(\mathfrak{sp}(4))$ is useful for the study of tensor product representations in which one tensor factor is an oscillator representation of $\mathfrak{sp}(4)$; however, the authors stress that there is no requirement in general for the codomain algebra to contain~$U(\mathfrak{sp}(4))$ in order to study or apply reduction algebras. We observe that \eqref{eq:sp4-oscillator} does not yield an injective algebra homomorphism from $U(\mathfrak{sp}(4))$ to $A_{2}$ as, amongst many relations, $b_{12}c_{12} - a_{12}a_{21} + a_{11} - \frac{1}{2} = 0$. Yet one can study the associated reduction algebra, a topic we review next.
\end{Remark}

\subsection{Reduction algebras}\label{subsec:ReductionAlgebras}
There is extensive literature \cite{ashtonRmatrixMickelssonAlgebras2015,hartwigGhostCenterRepresentations2023,khoroshkinMickelssonAlgebrasZhelobenko2008,herlemontRingsHdeformedDifferential2017,vandenhomberghNoteMickelssonStep1975,zhelobenkoExtremalProjectorsGeneralized1989,zhelobenkoHypersymmetriesExtremalEquations1997} on reduction algebras and their applications.
We provide a brief overview of constructions/perspectives and related notions for when $\mathfrak{g}$ is a~complex finite-dimensional reductive Lie algebra $\mathfrak{g}=\mathfrak{g}_{-}\oplus\mathfrak{h}\oplus\mathfrak{g}_{+}$ with fixed (maximal) nilpotent subalgebras $\mathfrak{g}_{\pm}$ and Cartan subalgebra $\mathfrak{h}$. See \cite[Section~2]{hartwigDiagonalReductionAlgebra2022c} for details in the super case.

{\bf Subquotient algebra.} Let $A$ be an associative algebra and $\Xi\colon U(\mathfrak{g})\to A$ a map of associative algebras. In particular, the map $\Xi$ makes $A$ into a $U(\mathfrak{g})$-bimodule via multiplication. Let~${I=A\cdot \mathfrak{g}_+}$ (written as $A\mathfrak{g}_+$ from hereon) be the left ideal in $A$ generated by the image of~$\mathfrak{g}_+$ under $\Xi$ composed with the natural inclusion. The reduction algebra $Z = Z(A,\mathfrak{g},\Xi)$ equals~${N_{A}(I)/I}$, where $N_{A}(I)$ is the normalizer of $I$ in $A$. To clarify, the \emph{normalizer} $N_{A}(I)$ of the left ideal $I$ in $A$ is the maximal subalgebra of $A$ with respect to containment of $I$ as a~two-sided ideal. Hence as an $A$-subquotient the reduction algebra $Z$ inherits its product from $A$. Note that normalizers may be termed \emph{idealizers} in general ring theory \cite[specifically Section~2 for historical context]{jacobsonNonCommutativePolynomialsCyclic1934}.

\begin{Remark}
	As mentioned in the introduction, the necessary data is encoded in a Lie algebra map $\Xi\colon \mathfrak{g} \rightarrow A$. Reduction algebras may be expressed as $Z(\mathfrak{t},\mathfrak{g})$ (respectively, $Z(A, \mathfrak{g})$) if the map $\mathfrak{g} \rightarrow U(\mathfrak{t})$ (respectively, $\mathfrak{g} \rightarrow A$) is given in context.
\end{Remark}

{\bf Opposite endomorphism algebra.} As $(\End_{A}(A))^{\mathrm{op}}$ is isomorphic to $A$, one might hope for a natural realization of the $A$-sub\-quo\-tient $Z$ as endomorphisms of an $A$-module.
Indeed, since $n \in N_{A}(I)$ implies $In \subset I$, right multiplication by $N_{A}(I)$ defines a right action on $A/I$ such that $I \subset \Ann_{N_{A}(I)}(A/I)$. It follows that $A/I$ is a right $Z$-module, in fact, an $(A,Z)$-bimodule. Then $\End_{A}(A/I)$ is a left $Z$-module isomorphic to the left $Z$-module $Z$, particularly since endomorphisms of $A/I$ preserve $Z$ and $1+I \in Z$ generates $A/I$ as an $A$-module. Now consider maps $1+I \mapsto n+I$, $n \in N_{A}(I)$, which are precisely the elements of $\End_{A}(A/I)$ by the previous statement. Then opposite composition shows that the reduction algebra $Z$ is isomorphic to $Z_{\text{maps}} = (\End_{A}(A/I))^{\mathrm{op}}$, the opposite algebra of left $A$-module endomorphisms on $A/I$.

{\bf Invariants.}
For a left $A$-module $V$, the subspace of $\mathfrak{g}_{+}$-invariants is $V^{+} = \{{v \in V \mid ev = 0} \allowbreak \text{for all } e\in \mathfrak{g}_{+} \}$.
Then since $N_{A}(I)V^{+} \subset V^{+}$ and $IV^{+} = \{0\}$, there is a natural action on~$V^{+}$ by~$Z$ on the left. Taking the case where $V = A/I$, we have that $Z = (A/I)^{+}$ since ${\mathfrak{g}_{+}(a + I) = \{0 + I\}}$ if and only if $Ia \subset I$, the defining condition for $a \in N_{A}(I)$. We say then that $A/I$ is the universal $\mathfrak{g}_{+}$-highest-weight $A$-module in the sense that given a pair $(W,w)$ of a~left $A$-module $W$ containing a (singular) vector $w$ (such as any $\mathfrak{g}_{+}$-highest-weight vector), there exists a unique $A$-module homomorphism $\Psi_{w}\colon A/I \rightarrow W$ sending $1+I$ to $w$. Evaluation at~${1+I \in (A/I)^{+}}$ determines $\Psi_{w}$ completely; thus, the left $Z$-module $\Hom_{A}(A/I,W)$ and $W^{+}$ are isomorphic
as left $Z$-modules, as in the special case in the previous paragraph.

\begin{Remark}
	Recasting $W^{+}$ as solutions to a system of equations puts $Z$ as a symmetry algebra of the solution space. In particular, reduction algebras are useful in determining solutions to so-called extremal equations \cite{zhelobenkoHypersymmetriesExtremalEquations1997}.
\end{Remark}

{\bf Coinvariants.}
There is a dual construction of reduction algebras based on $\mathfrak{g}_{-}$-coinvariants, the elements of $V_{-} = V/\mathfrak{g}_{-}V$ for a left $A$-module $V$ (taking into account the action of $\mathfrak{g}_{-}$ induced by $\Xi$).
We begin with the quotient space $(A/I)_{-} = (A/I)/\mathfrak{g}_{-}(A/I)$ that comprises the $\mathfrak{g}_{-}$-coinvariants of $A/I$. Note that $A/I$ being a right $Z$-module and $A$ being a $U(\mathfrak{g})$-bimodule implies $(A/I)_{-}$ is a right $Z$-module. By multiple applications of the fundamental theorem of homomorphisms and setting the right ideal $\mathfrak{g}_{-}A$ equal to $J$, we identify $(A/I)_{-}$ with the double coset space $J\reflectbox{$/$}A/I = A/\Romanbar{II}$, where $\Romanbar{II}$ is the sum of abelian groups $I + J$. Explicitly, the mapping sending $(a + I) + \mathfrak{g}_{-}(A/I) \in (A/I)_{-}$ to $a + (I + J) \in a + \Romanbar{II}$ is an isomorphism of vector spaces. Then we have the composition $\nu = \pi\circ\iota\colon Z \rightarrow A/\Romanbar{II}$ of right $Z$-modules from the natural inclusion~${\iota\colon Z\rightarrow A/I}$ and the natural surjection $\pi\colon A/I \rightarrow A/\Romanbar{II}$.

{\bf Extremal projector.} %\label{subsubsec:ExtremalProjector}
The use of $\Romanbar{II}$ for the sum $I+J$ of abelian groups is suggestive as the resulting quotient $A/\Romanbar{II}$ of $\mathfrak{g}_{-}$-coinvariants has the structure of an algebra isomorphic to $Z$ via the extremal projector.
Extremal projectors were introduced in \cite{lowdinAngularMomentumWavefunctions1964} and subsequently studied by Asherova, Smirnov, Tolstoy \cite{asherovaProjectionOperatorsSimple1973}.
Each finite-dimensional reductive Lie algebra $\mathfrak{g}$ over the complex numbers admits a unique extremal projector $P_{A,\mathfrak{g},\Xi} = P_{\mathfrak{g}} = P_{\mathfrak{g}_{+}} = P$ (varying only in expression) that projects the universal $\mathfrak{g}_{+}$-highest-weight $A$-module $A/I$ (more generally, an $A$-module $V$) onto the space of invariants $Z$ (respectively, $V^{+}$) along $J$ (respectively, $\mathfrak{g}_{-}V)$ when the input data $(A,\mathfrak{g},\Xi)$ satisfies sufficient conditions:
\begin{enumerate}[label=(C\arabic*)]\itemsep=0pt
	\item With the adjoint action, the Lie algebra $\mathfrak{g}$ acts reductively on $A$ via $\Xi$. \label{EPconditons:1:adjoint}
	\item The left action of $\mathfrak{g}_{+}$ on $A$ (or on $V$) is locally finite. \label{EPconditons:2:locally-finite}
	\item The image $\Xi(h-n)$ is invertible in $A$ for any coroot $h$, integer $n$ (\emph{coroot condition}). \label{EPconditons:3:coroot-condition}
\end{enumerate}
Using the categorical language of \cite{andreotti1957generalites}, we have $\nu^{-1} = (\pi \circ \iota)^{-1}$ is the $(A,Z)$-bimodule map coming by considering first the astriction of $P$ to $A/\Romanbar{II}$ (descending $P$ to $A/\Romanbar{II}$) before requiring the coastriction of the resulting map to its image, which is equal to $\operatorname{Im}(P) = Z$ (shrinking the codomain to $Z$): \smash{$A/I \xrightarrow{\pi} A/\Romanbar{II} \xrightarrow{\nu^{-1}} Z \xrightarrow{\iota} A/I$}. Then, as noted by Khoroshkin and Ogievetsky, we get the \emph{double coset algebra}:

\begin{Proposition}[{\cite[equation (3.7)]{khoroshkinMickelssonAlgebrasZhelobenko2008}}]\label{thm:double-coset-algebra}
	Under conditions {\rm\ref{EPconditons:1:adjoint}}, {\rm\ref{EPconditons:2:locally-finite}}, and {\rm\ref{EPconditons:3:coroot-condition}},
	there is an associative binary product $\diamond$ $($``diamond product''$)$ on $A/\II$, defined by
	\[\bar{x} \diamond \bar{y} = \pi(xP(y+I)), \qquad \bar{x} = x+ \II,\thinspace \bar{y} = y + \II\thinspace \in A/\emph{\Romanbar{II}},\]
	which yields $\nu^{-1}$ as an isomorphism of associative algebras: $A/\emph{\Romanbar{II}} \cong Z$.
	
	Note: We use $xPy + \II$ to mean the right-hand side of the diamond product $\bar{x} \diamond \bar{y}$ in $A/\II$.
\end{Proposition}

\begin{Remark}[on the diamond product]
	For $\nu^{-1}$ to be an algebra isomorphism, we require a~product $\diamond$ such that $\nu^{-1}(\bar{x} \diamond \bar{y}) = \nu^{-1}(\bar{x})\nu^{-1}(\bar{y}) \in Z$. Then necessarily ${\bar{x} \diamond \bar{y} = \bar{x}\nu^{-1}(\bar{y}) \in A/\Romanbar{II}}$, where the right-hand side is interpreted through the right $Z$-module structure on $A/\Romanbar{II}$.
	Now considering the left $A$-module structure on $A/I$ and the right $Z$-module map $\pi$, we have
\[
\pi(xP(y+I)) = \pi(x(1+I)P(y+I)) = \pi((x+I)P(y+I)) = \bar{x}\nu^{-1}(\bar{y}).
\]
\end{Remark}

{\bf Ring of dynamical scalars.}
To ensure the coroot condition is satisfied, we follow Zhe\-lobenko \cite{zhelobenkoHypersymmetriesExtremalEquations1997} to enact either of the equivalent steps of localizing the associative algebra $A$ as input data of a reduction algebra or localizing the resulting reduction algebra, whose non-localized version we refer to as the \emph{Mickelsson algebra}. The Mickelsson algebra may be regarded as a subalgebra of the corresponding (localized) reduction algebra. We also cite Zhelobenko \cite[equation (4.16)]{zhelobenko1985gelfand} for the method of obtaining elements of the Mickelsson algebra from the localized reduction algebra by clearing denominators---the elements obtained by this method are called \emph{normalized elements}. We will use normalized generators of $Z$ in specific calculations.

In localizing $A$, we extend the complex scalars to \emph{dynamical scalars}, which are (non-evaluated) rational functions in $\mathfrak{h}$. Precisely, consider the denominator set $S$ generated by $\{h + n \mid h \allowbreak\text{a coroot}, n \in \mathbb{Z}\}$, and write the localized $A$ as $\mathscr{A} = S^{-1}U(\mathfrak{h}) \otimes_{U(\mathfrak{h})} A$. Then the reduction algebra associated to the data $(\mathscr{A},\mathfrak{g},\Xi)$ (where we use the same name for enlarging the codomain of $\Xi$ from $A$ to $\mathscr{A}$) is identified with the localization of $Z(A, \mathfrak{g}, \Xi)$, namely, $S^{-1}U(\mathfrak{h}) \otimes_{U(\mathfrak{h})} Z$.

In general, computing the normalizer of a left ideal is unwieldy, and the Mickelsson algebra is not a finitely generated $\mathbb{C}$-algebra. Let $R$ be the ring of dynamical scalars $S^{-1}U(\mathfrak{h})$. The extension of the base ring to $R$ yields a finitely-generated $R$-ring \cite[Definition 2.1]{bohmHopfAlgebroids2009} $Z(\mathscr{A}, \mathfrak{g}, \Xi)$.

The thesis \cite{herlemontDifferentialCalculusMathbf2017} of Herlemont (supervised by Ogievetsky) expands on the connections between extremal projectors and reduction algebras. A more geometric treatment is found in \cite{sevostyanovGeometricMeaningZhelobenko2013}. We proceed with the details for the construction of the differential reduction algebra of $\mathfrak{sp}(4)$.

\section[Constructing a differential reduction algebra of sp(4)]{Constructing a differential reduction algebra of $\boldsymbol{\mathfrak{sp}(4)}$}\label{sec:sp4reduction}
From hereon, set $A = A_{2} \otimes U(\mathfrak{sp}(4))$ and recall the algebra map $\zeta$ from \eqref{eq:Homomorphism-sp4toA}. In this section, we write the extremal projector for $\mathfrak{sp}(4)$ in order to provide a finite presentation of the differential reduction algebra $D(\mathfrak{sp}(4))$ defined within.

\subsection[The differential reduction algebra of sp(4)]{The differential reduction algebra of $\boldsymbol{\mathfrak{sp}(4)}$}\label{subsec:DRAsp4}

Set the ring $R$ of dynamical scalars to be
\begin{equation}\label{eq:R-ring-of-dynamical-scalars}
	R=S^{-1}\zeta(U(\mathfrak{h})),
\end{equation}
where $S\subset A$ is the multiplicative submonoid generated by
$\{\zeta(h_\gamma)+n\mid \gamma\in\Phi_+, n \in \mathbb{Z}\}$. The set $S$ is an Ore denominator set in $A$, so we may define $\mathscr{A}$ as the localization
$\mathscr{A} = S^{-1}A$.
Composing the canonical inclusion $A\to\mathscr{A}$ with $\zeta$, we regard $\zeta$ as an algebra map $\zeta\colon U(\mathfrak{sp}(4))\to\mathscr{A}$.
Then~$(\mathscr{A},\mathfrak{sp}(4),\zeta)$ satisfies the conditions \ref{EPconditons:1:adjoint}, \ref{EPconditons:2:locally-finite}, and \ref{EPconditons:3:coroot-condition}.

Let
$I_+ = \mathscr{A}E_\alpha+\mathscr{A}E_\beta$
be the left ideal of $\mathscr{A}$ generated by the image of the positive nilpotent part of $\mathfrak{sp}(4)$ and $I_-=F_\alpha\mathscr{A}+F_\beta\mathscr{A}$ the right ideal generated by the image of the negative nilpotent part of $\mathfrak{sp}(4)$. Then as in Section \ref{subsec:ReductionAlgebras}, the space of double cosets with respect to the sum $\II=I_-+I_+$ in $\mathscr{A}$ is
$D(\mathfrak{sp}(4))=\mathscr{A}/\II$.
Moreover, $D(\mathfrak{sp}(4))$ has the structure of an algebra under the diamond product. Following \cite[Theorem 2]{zhelobenkoExtremalProjectorsGeneralized1989}, the extremal projector for $\mathfrak{sp}(4)$ is
\[
	P=P_\beta P_{\beta+\alpha}P_{\beta+2\alpha}P_\alpha = P_\alpha P_{\beta+2\alpha}P_{\beta+\alpha}P_{\beta},
\]
where, using the shifted coroots $H_\gamma$ for $\gamma \in \Phi_{+}$,
\[
	P_\gamma = 1-\frac{1}{H_\gamma+2}F_\gamma E_\gamma + \cdots.
\]
And
$\bar a\diamond \bar b = aPb + \II$ for all $\bar a=a+\II$, $ \bar b=b+\II$.

We call $D(\mathfrak{sp}(4))$ the \emph{differential reduction algebra of $\mathfrak{sp}(4)$}, and $D(\mathfrak{sp}(4))$ is an associative algebra isomorphic to the
reduction algebra $Z(\mathscr{A},\mathfrak{sp}(4))=N_{\mathscr{A}}(I_+)/I_+$.
We introduce the following Weyl-algebra like elements.
For $a\in\{\partial_1,\partial_2,x_1,x_2\}$, we put
\[%\label{eq:tilde-and-bar-definitions}
	\tilde a=a\otimes 1\in \mathscr{A}, \qquad \bar a=\tilde a+ \II \in D(\mathfrak{sp}(4)).
\]

\begin{Lemma}\label{lem:Dsp4-anti-automorphism}
	The algebra $D(\mathfrak{sp}(4))$ carries an involutive anti-automorphism $\Theta$ determined by
	\[
		\Theta(\bar x_i) = \bar \partial_i,\qquad i=1,2.
	\]
\end{Lemma}

\begin{proof}
	The Chevalley involution $\tau$ of $\mathfrak{sp}(4)$ extends to an involutive algebra anti-automorphism of $U(\mathfrak{sp}(4))$, also denoted $\tau$. Recalling $\vartheta$ from \eqref{eq:vartheta}, and $\zeta$ from \eqref{eq:Homomorphism-sp4toA},
	we have
	\[
		\zeta(\tau(a)) = (\vartheta\otimes \tau)\circ \zeta(a),\qquad a\in\mathfrak{sp}(4).
	\]
	The anti-automorphism $\vartheta\otimes\tau$ of $A_2\otimes U(\mathfrak{sp}(4))$ uniquely extends to the localization $\mathscr{A}$ and preserves the double coset $\II$ and fixes the extremal projector $P$, hence induces an anti-automor\-phism~$\Theta$ of the required form.
\end{proof}

Lemma \ref{lem:Dsp4-anti-automorphism} allows us to cut computations almost in half.

\subsection[A computational lemma for D(sp(4))]{A computational lemma for $\boldsymbol{D(\mathfrak{sp}(4))}$}\label{subsec:ComputationalLemma}
We provide some computations that will be needed in proof of Theorem \ref{thm:presentation}.
In $\mathscr{A}$, we have for any positive root $\gamma$,
\begin{gather}\label{eq:P_gamma}
	aP_\gamma b \equiv
	ab+aF_\gamma\frac{-1}{H_\gamma}E_{\gamma}b + \cdots
	\equiv ab+[a,F_\gamma]\frac{-1}{H_\gamma}[E_{\gamma},b] + \cdots\quad \mod{\II},
\end{gather}
where $\cdots$ indicate higher-order terms from the extremal projector. However, we will only need the constant and linear terms in the calculations of this paper.

\begin{Lemma} \label{lem:aPb-congruences}
	We have the following congruences modulo $\II$:
	\begin{gather}%\label{cong:x2Pp2}
		y P\widetilde{x}_{1} \equiv y\widetilde{x}_{1}\qquad \text{ for all}\quad y \in \mathscr{A}\label{cong:p1Px1},\\
		\widetilde{\partial}_2P\widetilde{x}_2 \equiv \widetilde{\partial}_2\widetilde{x}_2+\frac{-1}{H_\alpha+1}\widetilde{\partial}_1\widetilde{x}_1\label{cong:p2Px2},\\
		\widetilde{x}_2P\widetilde{\partial}_2 \equiv -1 + \frac{H_\beta}{(H_\beta+1)(H_{\beta+\alpha}+1)}\widetilde{\partial}_1\widetilde{x}_1+\left(1+\frac{1}{H_\beta+1}\right)\widetilde{\partial}_2\widetilde{x}_2,\nonumber\\
		\widetilde{x}_1P\widetilde{\partial}_1 \equiv -1+\frac{1}{H_\alpha+1}+\left(1+\frac{H_\alpha H_{\beta+\alpha}+H_{\beta+2\alpha}+1}{(H_\alpha+1)(H_{\beta+\alpha}+1)(H_{\beta+2\alpha}+1)}\right)\widetilde{\partial}_1\widetilde{x}_1\nonumber\\
		 \phantom{\widetilde{x}_1P\widetilde{\partial}_1 \equiv }{} + \frac{H_\alpha-H_{\beta+\alpha}-2}{(H_\alpha+1)(H_{\beta+\alpha}+1)}\widetilde{\partial}_2\widetilde{x}_2. \label{cong:x1Pp1}
	\end{gather}
\end{Lemma}

\begin{proof}
	The congruence \eqref{cong:p1Px1} follows immediately from the fact that $[E_\gamma,\widetilde{x}_1]=0$ for all positive roots $\gamma$.
	For the proof of \eqref{cong:p2Px2}, we have
	\begin{align*}
		\widetilde{\partial}_2 P\widetilde{x}_2 &\equiv \widetilde{\partial}_2 P_\alpha P_{\beta+2\alpha} P_{\beta+\alpha} P_\beta \widetilde{x}_2 &\text{definition of $P$}\\
		&\equiv \widetilde{\partial}_2 P_\alpha \widetilde{x}_2 &\text{by $[E_{\beta+k\alpha},\widetilde{x}_i]=0$}\\
		&\equiv \widetilde{\partial}_2\widetilde{x}_2+\bigl[\widetilde{\partial}_2,F_\alpha\bigr]\frac{-1}{H_\alpha}[E_\alpha,\widetilde{x}_2]&\text{by \eqref{eq:P_gamma}}\\
		&\equiv \widetilde{\partial}_2\widetilde{x}_2 + \widetilde{\partial}_1\frac{-1}{H_\alpha}\widetilde{x}_1&\text{by $[\widetilde{\partial}_2,F_\alpha]=\widetilde{\partial}_1$}\\
		&\equiv \widetilde{\partial}_2\widetilde{x}_2+\frac{-1}{H_\alpha+1}\widetilde{\partial}_1\widetilde{x}_1&\text{by $H_\alpha\widetilde{\partial}_1=\widetilde{\partial}_1(H_\alpha-1)$}.
	\end{align*}
	We have proved \eqref{cong:p2Px2}. The remaining computations are similar in nature but also rely on the~$\mathfrak{sp}(4)$ Lie algebra relations.
\end{proof}

\subsection[Finite presentation of D(sp(4))]{Finite presentation of $\boldsymbol{D(\mathfrak{sp}(4))}$}\label{subsec:PresentationDRAsp4}
\label{sec:presentation}

We cite \cite{khoroshkinDiagonalReductionAlgebras2010} as a reference on the discussion of producing generators of reduction algebras generally according to a weight-basis ordering and the use of extremal projectors.

\begin{Lemma}[generators]\label{lemma:finite-generators}
	As a left $R$-module, $D(\mathfrak{sp}(4))$ is free with basis
	\[
		\bigl\{
		\bar \partial_1^{\diamond a} \diamond
		\bar \partial_2^{\diamond b} \diamond
		\bar x_2^{\diamond c} \diamond
		\bar x_1^{\diamond d}
		\mid a,b,c,d\in\mathbb{Z}_{\geq 0}\bigr\},
	\]
	where an exponent $\diamond k$ signifies a $k$-fold diamond product.
	In particular, as an $R$-ring, $D(\mathfrak{sp}(4))$ is generated by the four elements $\bar{x}_{1}$, $\bar{x}_{2}$, $\bar{\partial}_{1}$, $\bar{\partial}_{2}$.
\end{Lemma}

\begin{proof}
	Let $\mathfrak{g}=\mathfrak{sp}(4)$ with triangular decomposition $\mathfrak{g}=\mathfrak{g}_-\oplus\mathfrak{h}\oplus \mathfrak{g}_+$.
	By the PBW theorem, $U(\mathfrak{g})$ is a free left $U(\mathfrak{h})$-module;
	therefore,
	we have an isomorphism of left $U(\mathfrak{h})$-modules relying on the usual comultiplication $\Delta(x) = x \otimes 1 + 1\otimes x$ on $U(\mathfrak{g})$,
	\begin{gather*}
		U(\mathfrak{h})\otimes U(\mathfrak{g}_-)\otimes A_{2} \otimes U(\mathfrak{g}_+) \to A = A_2\otimes U(\mathfrak{g}), \\
		1\otimes w \otimes 1 \otimes 1 \mapsto w\otimes 1,\qquad w\in A_2,\qquad
		h\otimes 1\otimes 1\otimes 1 \mapsto (\zeta\otimes \Id)\circ\Delta(h),\qquad h\in U(\mathfrak{h}),\\
		1\otimes 1\otimes f\otimes 1 \mapsto (\zeta\otimes \Id)\circ\Delta(f),\qquad f\in U(\mathfrak{g}_-),\\
		1\otimes 1\otimes 1\otimes e \mapsto (\zeta\otimes \Id)\circ\Delta(e),\qquad e\in U(\mathfrak{g}_+),
	\end{gather*}
	where $\zeta$ is the algebra map $U(\mathfrak{g})\to A_2\otimes U(\mathfrak{g})$ from \eqref{eq:Homomorphism-sp4toA}.
	In particular, $A$ is free as a left $U(\mathfrak{h})$-module. Consequently, the localized algebra $\mathscr{A}=S^{-1}A$ is free as a left $R$-module.
	The preceding sentences imply that $\mathscr{A}$ has a basis as a left $R$-module consisting of monomials $FWE$,
	where $F$ and $E$ are images under $\zeta$ of PBW monomials (with respect to any choice of ordered bases) for~$U(\mathfrak{g}_-)$ and $U(\mathfrak{g}_+)$, respectively, and \smash{$W=\bigl(\tilde{\partial}_{1}\bigr)^a\bigl(\tilde{\partial}_{2}\bigr)^b (\tilde{x}_{2})^c(\tilde{x}_{1})^d$} is the image of the corresponding element without tildes in the $\mathbb{C}$-basis $\bigl\{(\partial_{1})^a(\partial_{2})^b (x_{2})^c(x_{1})^d
	\mid a,b,c,d \in \mathbb{Z}_{\geq 0} \bigr\} \subset A_{2}$ for the second Weyl algebra.
	A subset of the $R$-basis $\{FWE\}$ for $\mathscr{A}$ just described is an $R$-basis for the space $\II=\zeta(\mathfrak{g}_-)\mathscr{A}+\mathscr{A}\zeta(\mathfrak{g}_+)$, namely the set of those $FWE$ with either $F \neq \zeta(1)$ or~${E \neq \zeta(1)}$ (or $\zeta(z)$ if $1$ is replaced in the PBW basis with another complex unit $z$). Therefore, we immediately get a basis for $D(\mathfrak{g})=\mathscr{A}/\II$ as a left $R$-module
	\[%\label{R-basis-for-double-cosets}
		\bigl\{\bigl(\tilde{\partial}_{1}\bigr)^a\bigl(\tilde{\partial}_{2}\bigr)^b (\tilde{x}_{2})^c(\tilde{x}_{1})^d + \II\mid a,b,c,d\in\mathbb{Z}_{\geq 0}\bigr\}.
	\]
	
	We now order the single-element monomials $\tilde{\partial}_{1}$, $\tilde{\partial}_{2}$, $\tilde{x}_{2}$, $\tilde{x}_{1}$, writing $x \prec y$ for ``$x$ precedes~$y$'', based on their $\zeta(\mathfrak{h})$-weights in $\mathscr{A}$ (where $\prec$ is compatible with the standard root order on $\zeta(\mathfrak{h})^{\ast}$). Now $x_{1}$ is a highest weight vector since $[E_\alpha, x_1] = 0$ and $[E_\beta, x_1]=0$. Thus~$\partial_{1}$ is a lowest weight vector using $\Theta$. We further have
	$[F_\alpha, x_1] \in\mathbb{C} x_2$ and $[F_\beta, x_1]=0$; $[E_{\alpha}, \partial_{2}] = 0$ and $[E_{\beta}, \partial_{2}] \in \mathbb{C} x_{2}$;
	therefore, we order $\partial_{1}$, $\partial_{2}$, $x_{2}$, $x_{1}$ by
	\begin{equation}\label{eq:order-xandpartial}
		1 \prec \partial_1 \prec \partial_2 \prec x_2 \prec x_1
	\end{equation}
	and equip the corresponding set of monomials of the form
	$\bigl(\tilde{\partial}_{1}\bigr)^a\bigl(\tilde{\partial}_{2}\bigr)^b (\tilde{x}_{2})^c(\tilde{x}_{1})^d$
	with the lexicographical ordering induced by \eqref{eq:order-xandpartial}.
	Then, by definition of the extremal projector and the Weyl algebra relations \eqref{eq:weylrels},
	\begin{align*}
		\overline{\partial_1^a \partial_2^b x_2^c x_1^d }\diamond \overline{\partial_1^e \partial_2^f x_2^g x_1^h }={}& \bigl(\tilde{\partial}_{1}\bigr)^a \bigl(\tilde{\partial}_{2}\bigr)^b (\tilde{x}_{2})^c (\tilde{x}_{1})^d \bigl(\tilde{\partial}_{1}\bigr)^e \bigl(\tilde{\partial}_{2}\bigr)^f (\tilde{x}_{2})^g (\tilde{x}_{1})^h \\
& +\text{lower order terms} + \Romanbar{II}\\
={}& \bigl(\tilde{\partial}_{1}\bigr)^{a+e} \bigl(\tilde{\partial}_{2}\bigr)^{b+f} (\tilde{x}_{2})^{c+g} (\tilde{x}_{1})^{d+h} + \text{lower order terms} + \Romanbar{II}; \end{align*}
	and consequently, by induction,
	\[
	\bar \partial_1^{\diamond a} \diamond
	\bar \partial_2^{\diamond b} \diamond
	\bar x_2^{\diamond c} \diamond
	\bar x_1^{\diamond d} = \bigl(\tilde{\partial}_{1}\bigr)^a \bigl(\tilde{\partial}_{2}\bigr)^b (\tilde{x}_{2})^c (\tilde{x}_{1})^d + \text{lower order terms} + \Romanbar{II}.\]
	For example, see \eqref{cong:p2Px2},
	\[\bar \partial_2 \diamond \bar x_2 = \partial_2x_2 + f(H_\alpha,H_\beta) \partial_1 x_1 + g(H_\alpha,H_\beta)1 + \Romanbar{II},\]
	where $f$ and $g$ are dynamical scalars in $R$, particularly, $f = f(H_\alpha) = \frac{-1}{H_{\alpha}+1}$ and $g = 1$.
	
	The argument above proves that the two sets
	\smash{$\bigl\{ \bigl(\tilde{\partial}_{1}\bigr)^a \bigl(\tilde{\partial}_{2}\bigr)^b (\tilde{x}_{2})^c (\tilde{x}_{1})^d +\II \mid a,b,c,d \in \mathbb{Z}_{\geq 0} \bigr\}$} and
	$\bigl\{\bar \partial_1^{\diamond a} \diamond
	\bar \partial_2^{\diamond b} \diamond
	\bar x_2^{\diamond c} \diamond
	\bar x_1^{\diamond d} \mid a,b,c,d \in \mathbb{Z}_{\geq 0} \bigr\}$ are related by a triangular matrix with $1$'s on the diagonal. Therefore, the latter is a basis for $\mathscr{A}/\II$ as a left $R$-module. Thus $D(\mathfrak{sp}(4))$ is the smallest $R$-ring containing $\bar{x}_{1}$, $\bar{x}_{2}$, $\bar{\partial}_{1}$, $\bar{\partial}_{2}$.
\end{proof}

\begin{Theorem}[finite presentation]\label{thm:presentation}
	The $R$-ring $D(\mathfrak{sp}(4))$ is generated by $\bar x_1$, $\bar \partial_1$, $ \bar x_2$, $\bar\partial_2$ subject to the following finite set of relations:
	\begin{subequations}\label{eq3.7}		\begin{gather}%\label{eq:rel-p1-H}\label{eq:rel-x2-H}
			\label{eq:rel-x1-H}
			\bar x_1 H_\alpha = (H_\alpha-1)\bar x_1,\qquad			
			\bar x_1 H_\beta = H_\beta \bar x_1, \\
			\bar \partial_1 H_\alpha = (H_\alpha+1)\bar \partial_1,\qquad
						\bar \partial_1 H_\beta = H_\beta \bar \partial_1, \nonumber
\\
			\bar x_2 H_\alpha = (H_\alpha+1)\bar x_2,
			\qquad
			\bar x_2 H_\beta = (H_\beta -1)\bar x_2,\nonumber\\
			\label{eq:rel-p2-H}
			\bar \partial_2 H_\alpha = (H_\alpha-1)\bar \partial_2,
			\qquad
			\bar \partial_2 H_\beta = (H_\beta +1)\bar \partial_2,\\
			\label{eq:rel-x1x2}
			\bar x_1\diamond\bar x_2 = \Big(1+\frac{1}{H_\alpha+1}\Big)\bar x_2\diamond \bar x_1,\qquad
			\bar \partial_2\diamond\bar \partial_1 = \bar \partial_1\diamond \bar \partial_2 \Big(1+\frac{1}{H_\alpha+1}\Big),\\
			\label{eq:rel-x1p2}
			\bar x_1 \diamond \bar\partial_2 =
			\Big(1+\frac{1}{H_{\beta+\alpha}+1}\Big)
			\bar \partial_2 \diamond \bar x_1,\qquad
			\bar x_2\diamond\bar\partial_1 =
			\bar \partial_1 \diamond \bar x_2
			\Big(1+\frac{1}{H_{\beta+\alpha}+1}\Big),
\\
			\label{eq:rel-x1p1}
			\bar x_1\diamond\bar \partial_1 = -1+\frac{1}{H_\alpha+1}+f_{11}\bar\partial_1\diamond\bar x_1+f_{12}\bar\partial_2\diamond\bar x_2,
			\\
			\label{eq:rel-x2p2}
			\bar x_2\diamond\bar \partial_2 =
			-1+f_{21}\bar \partial_1\diamond \bar x_1 +f_{22}\bar\partial_2\diamond\bar x_2,
		\end{gather}\end{subequations}
	where $f_{ij}=f_{ij}(H_\alpha,H_\beta)\in R$ are given by
	\begin{alignat}{3}
			& f_{11}=\frac{(a+1)(a-1)(b+1)}{a^2b}, \qquad&&
			f_{12}=\frac{-(d+2)}{ac},&\nonumber\\
			& f_{21}=\frac{a(d-1) + c(d+1)}{acd},\qquad&&
			f_{22}=\frac{d+1}{d},&\label{eq:abcdfij}
	\end{alignat}
	and we put $a=H_\alpha+1$, $ b=H_{\beta+2\alpha}+1$, $c=H_{\beta+\alpha}+1$, $d=H_\beta+1$.
\end{Theorem}

\begin{proof}
	By Lemma \ref{lemma:finite-generators}, we are left to show \eqref{eq3.7} hold. We offer computational methods for the different type of relations: There are relations involving only the $(R,R)$-bimodule structure, \eqref{eq:rel-x1-H}--\eqref{eq:rel-p2-H}; relations involving one diamond product, \eqref{eq:rel-x1x2} and \eqref{eq:rel-x1p2}; and those involving two diamond products, \eqref{eq:rel-x1p1} and \eqref{eq:rel-x2p2}.
	
	Relations \eqref{eq:rel-x1-H}--\eqref{eq:rel-p2-H} follow from the corresponding identities in $\mathscr{A}$. For example,
	\begin{align*}
		\bar x_1H_\alpha
		&
		= \tilde x_1H_\alpha+\II
		=(x_1\otimes 1)((x_1\partial_1-x_2\partial_2)\otimes 1+1\otimes h_\alpha)+\II
		\\
		&=(x_1\partial_1-x_2\partial_2-1)x_1\otimes 1+x_1\otimes h_\alpha+\II
		= (H_\alpha-1)\bar x_1,
	\end{align*}
	where we used the Weyl algebra relations of \eqref{eq:weylrels}.
	
	Next we consider relations involving one diamond product, namely, relations \eqref{eq:rel-x1x2} and \eqref{eq:rel-x1p2}. The application of $\Theta$ produces one half of the results. For example, consider \eqref{eq:rel-x1x2} and recall~\eqref{cong:p1Px1}.
	In particular, with $\equiv$ meaning congruence mod $\II$, we have
	\begin{equation}\label{eq:pf-rel-x1x2a}
		\tilde x_2 P \tilde x_1 \equiv \tilde x_2\tilde x_1 \mod{\II}.
	\end{equation}
	Similarly, none of $E_\beta$, $E_{\beta+\alpha}$, $E_{\beta+2\alpha}$ involve $\partial_2$, which implies that
$[E_\gamma, \tilde x_2] = 0$ for all $\gamma\in\Phi_+\setminus\{\alpha\}$,
	but
	\begin{equation*}
		[E_\alpha,\tilde x_2] = [x_1\partial_2\otimes 1+1\otimes e_\alpha, x_2\otimes 1] = x_1\otimes 1=\tilde x_1.
	\end{equation*}
	Likewise,
	$[\tilde x_1, F_\alpha] = -\tilde x_2$.
	Consequently, using \eqref{eq:rel-x1-H}, we have
	\begin{align}
			\tilde x_1 P \tilde x_2 &\equiv \tilde x_1P_\alpha P_{\beta+2\alpha}P_{\beta+\alpha}P_\beta \tilde x_2 \equiv \tilde x_1P_\alpha \tilde x_2
			\equiv \tilde x_1 \tilde x_2 - \frac{1}{H_\alpha+1} [\tilde x_1,F_{\alpha}] [E_{\alpha},\tilde x_2]\nonumber\\
			&\equiv \tilde x_1\tilde x_2 + \frac{1}{H_\alpha+1} \tilde x_2 \tilde x_1
			\equiv \frac{H_\alpha+2}{H_\alpha+1} \tilde x_2\tilde x_1.\label{eq:comp-x1px2}
		\end{align}
	Comparing \eqref{eq:pf-rel-x1x2a} and \eqref{eq:comp-x1px2}, the first relation in \eqref{eq:rel-x1x2} is proved; moreover, $\Theta$ yields the second relation in \eqref{eq:rel-x1x2}. Relation \eqref{eq:rel-x1p2} is proved in the same fashion with a simplification from $\tilde{x}_{1}P\tilde{\partial}_{2}$ to $\tilde{x}_{1}P_{\beta+\alpha}\tilde{\partial}_{2}$ and another application of \eqref{cong:p1Px1}.
	
	Lastly, we analyze the relations involving two diamond products and use Lemma \ref{lem:aPb-congruences}. Consider relation \eqref{eq:rel-x2p2}: Solving for $\widetilde{\partial}_i\widetilde{x}_i$ in \eqref{cong:p1Px1} and \eqref{cong:p2Px2}, we get
	\begin{align}
		\widetilde{\partial}_1\widetilde{x}_1+\II &= \bar\partial_1\diamond\bar x_1,
		\label{eq:p1x1=diamonds}\\
		\widetilde{\partial}_2\widetilde{x}_2+\II &= \bar\partial_2\diamond\bar x_2 + \frac{1}{H_\alpha+1}\bar\partial_1\diamond\bar x_1.
		\label{eq:p2x2=diamonds}
	\end{align}
	Substituting equations \eqref{eq:p1x1=diamonds} and \eqref{eq:p2x2=diamonds} into the right-hand side of equation \eqref{cong:p2Px2}, we arrive~at
	\begin{align*}
		\bar x_2\diamond\bar\partial_2 &= -1+\frac{H_\beta}{(H_\beta+1)(H_{\beta+\alpha}+1)}\bar\partial_1\diamond\bar x_1 + \frac{H_\beta+2}{H_\beta+1}\left(\bar\partial_2\diamond\bar x+2+\frac{1}{H_\alpha+1}\bar\partial_1\diamond\bar x_1\right) \\
		&=-1+\frac{H_\beta+2}{H_\beta+1}\bar\partial_2\diamond\bar x_2 + \frac{H_\beta(H_\alpha+1)+(H_\beta+2)(H_{\beta+\alpha}+1)}{(H_\alpha+1)(H_\beta+1)(H_{\beta+\alpha}+1)}\bar\partial_1\diamond\bar x_1\\
		&=-1+f_{22}\bar\partial_2\diamond\bar x_2 + f_{21}\bar\partial_1\diamond\bar x_1,
	\end{align*}
	proving relation \eqref{eq:rel-x2p2}. To be sure, substituting equations \eqref{eq:p1x1=diamonds} and \eqref{eq:p2x2=diamonds} into \eqref{cong:x1Pp1}, we obtain relation \eqref{eq:rel-x1p1}.
	Thus we have a finite presentation of $D(\mathfrak{sp}(4))$ as a ring over the dynamical scalars $R$. 	
\end{proof}

\begin{Remark}[quantum algebra and integral form]\label{Remark:integral-form}
	Reduction algebras are quantum deformations of associative algebras \cite{herlemontDifferentialCalculusHDeformed2017}. In our case, if we formally send all $H_\gamma\to\infty$ we obtain the usual Weyl algebra relations. More precisely, consider the extension of scalars $D^\ast(\mathfrak{sp}(4))=D(\mathfrak{sp}(4))\bigl[\hbar,\hbar^{-1}\bigr]\cong\mathbb{C}\bigl[\hbar,\hbar^{-1}\bigr]\otimes_\mathbb{C} D(\mathfrak{sp}(4))$. There is an ``integral form'' $D^0(\mathfrak{sp}(4))$ inside this algebra defined as the $\mathbb{C}[\hbar]$-subalgebra of $D^\ast(\mathfrak{sp}(4))$ generated by
	\[\check{H}_\alpha =\hbar H_\alpha,\quad \check{H}_\beta=\hbar H_\beta, \quad \bar x_1,\quad \bar x_2, \quad \bar\partial_1,\quad \bar\partial_2.\]
	Substituting $H_\gamma=\frac{1}{\hbar}\check{H}_\gamma$ in all relations above (and multiplying \eqref{eq:rel-x1-H}--\eqref{eq:rel-p2-H} by $\hbar$), we see that the quotient algebra of $D^0(\mathfrak{sp}(4))$ by the principal ideal $\langle \hbar \rangle$ is isomorphic to the Weyl algebra~$A_2(\mathbb{C})$ tensored (over $\mathbb{C}$) by the ring $R$.
\end{Remark}

\begin{Corollary}\label{cor:Dsp4-structure:domain}
	The ring $D(\mathfrak{sp}(4))$ is a domain $($i.e., there are no left or right zero-divisors$)$.
\end{Corollary}

\begin{proof}
	One may use Lemma \ref{lemma:finite-generators}
	to prove $D(\mathfrak{sp}(4))$ is a domain.
	Instead, we appeal to Remark~\ref{Remark:integral-form}.
	Since $D(\mathfrak{sp}(4))$ is a subalgebra of $D^\ast(\mathfrak{sp}(4))$, it suffices to show the latter is a domain.
	If~${x,y\in D^\ast(\mathfrak{sp}(4))}$ with $xy=0$, then $\hbar^kx, \hbar^ly \in D^0(\mathfrak{sp}(4))$ for sufficiently large positive integers~$k$,~$l$, and $\bigl(\hbar^kx\bigr)\bigl(\hbar^ly\bigr)=\hbar^{k+l}xy=0$. Thus it suffices to show that $D^0(\mathfrak{sp}(4))$ is a domain.
	The algebra $D^0(\mathfrak{sp}(4))$ has a filtration given by $\smash{D^0(\mathfrak{sp}(4))_{(k)}=D^0(\mathfrak{sp}(4))+D^0(\mathfrak{sp}(4))\hbar}+\cdots+D^0(\mathfrak{sp}(4))\hbar^k$ for $k\in\mathbb{Z}_{\geq 0}$, and the associated graded algebra is isomorphic to $D^0(\mathfrak{sp}(4))/\langle\hbar\rangle$. As mentioned in
	Remark \ref{Remark:integral-form}, the quotient $D^0(\mathfrak{sp}(4))/\langle\hbar\rangle$ is isomorphic to $R\otimes_\mathbb{C} A_2(\mathbb{C})$, which is a~domain. Therefore, $\operatorname{gr}D^0(\mathfrak{sp}(4))$ is a domain; thus $D^0(\mathfrak{sp}(4))$, and consequently, $D(\mathfrak{sp}(4))$ are domains.
\end{proof}

\section{Generalized Weyl algebras}\label{sec:GWA}
In the 1990s, Bavula \cite{bavula1992generalized} described a class of algebras that extends notions of the $n$-th Weyl algebra~$A_{n}$ and called them generalized Weyl algebras (GWAs).
GWAs include important examples from representation theory and ring theory, such as $U\big(\mathfrak{sl}(2)\big)$, $U_q\big(\mathfrak{sl}(2)\big)$, down-up algebras \cite{benkartAlgebras1998}, skew Laurent-polynomial rings, to name a few. The class is closed under taking tensor products and certain skew polynomial extensions.
See \cite{gaddisWeylAlgebraIts2023} for a recent survey on GWAs.

In \cite{van1976harisch}, GWA presentations are given for certain subalgebras of Mickelsson step algebras $S(\mathfrak{g}_{n+1},\mathfrak{g}_n)$, where $\mathfrak{g}_n=\mathfrak{sl}(n)$ or $\mathfrak{so}(n)$. In \cite{mazorchukAssociativeAlgebrasRelated2003}, certain reduction algebras were shown to be twisted GWAs. In this section, we recall the definition of GWAs and observe that $D(\mathfrak{sp}(4))$
is an example of an interesting subclass of GWAs that we term \emph{skew-affine} GWAs.

We will rely on normalized generators of $D(\mathfrak{sp}(4))$ and their relations. The following statement follows from Theorem \ref{thm:presentation}.
\begin{Proposition}[normalization]\label{Prop:Normalized-generators-relations}
	The normalized generators
	\[
		\hat x_1 = x_1,\qquad \hat x_2=(H_\alpha+2)\bar x_2,\qquad
		\hat \partial_1 = \bar \partial_1(H_\alpha+1)(H_{\beta+\alpha}+1),\qquad \hat \partial_2=\bar \partial_2(H_{\beta+\alpha}+1)
	\]
	satisfy
	\begin{equation}\label{eq:hat-relations}
		[\hat x_1,\hat x_2]_\diamond = \bigl[\hat \partial_1,\hat \partial_2\bigr]_\diamond =\bigl[\hat x_1, \hat \partial_2\bigr]_\diamond = \bigl[\hat x_2, \hat \partial_1\bigr]_\diamond=0,
	\end{equation}
	where $[a,b]_\diamond = a\diamond b-b\diamond a$ is the diamond commutator.
\end{Proposition}

\subsection[The differential reduction algebra D(sp(4)) is a generalized Weyl algebra]{The differential reduction algebra $\boldsymbol{D(\mathfrak{sp}(4))}$ is\\ a generalized Weyl algebra}\label{subsec:GWAsp4}
\begin{Definition}[GWA]\label{def:GWA}
	Let $\mathbb{K}$ be a field.
	Let $B$ be an associative $\mathbb{K}$-algebra, $n$ be a positive integer, $\sigma=(\sigma_1,\sigma_2,\dots,\sigma_n)\in\Aut_\mathbb{K}(B)^n$ an $n$-tuple of commuting $\mathbb{K}$-algebra automorphisms of $B$, $t=(t_1,t_2,\dots,t_n)\in Z(B)^n$ an $n$-tuple of elements of the center of $B$. The associated \emph{generalized Weyl algebra $($GWA$)$ of rank $($or degree$)$ $n$}, denoted $B(\sigma,t)$, is the $B$-ring generated by $X_1, Y_1, \dots, X_n, Y_n$ subject to the following relations, for $i,j=1,2,\dots,n$ and all $b\in B$:
	\begin{subequations}\label{eq:GWA-rels}
		\begin{gather}%	\label{eq:GWA-rels3}
			\label{eq:GWA-rels1}
			X_ib=\sigma_i(b)X_i,\qquad bY_i=Y_i\sigma_i(b),\\
			\label{eq:GWA-rels2}
			X_iX_j=X_jX_i,\qquad Y_iY_j=Y_jY_i,\qquad
			X_iY_j=Y_jX_i \qquad\text{if}\quad i\neq j,\\
			Y_iX_i=t_i,\qquad X_iY_i=\sigma_i(t_i).\nonumber
		\end{gather}
	\end{subequations}
\end{Definition}

\begin{Example}[revisiting the Weyl algebra]
	Taking $\mathbb{K}=\mathbb{C}$, $B=\mathbb{C}[u_1,u_2,\dots,u_n]$, $\sigma_i$ defined by $\sigma_i(u_j)=u_j-\delta_{ij}$, and $t_i=u_i$, the GWA $B(\sigma,t)$ is isomorphic to the $n$-th Weyl algebra $A_n(\mathbb{C})$ via $X_i\mapsto x_i$, $Y_i\mapsto \partial_i$.
\end{Example}

\begin{Proposition}[automorphisms]\label{Prop:Automorphisms}
	Let $B = R[t_{1},t_{2}]$, where $R$ is the ring of dynamical scalars defined in \eqref{eq:R-ring-of-dynamical-scalars}. The maps $\sigma_{1},\sigma_{2}\colon B \rightarrow B$ are automorphisms of $B$ when defined as follows
		\begin{gather*}
			%\label{eq:sigma-on-H-alpha}\label{eq:sigma-on-H-beta}
			\sigma_1(H_\alpha) =H_\alpha-1, \qquad \sigma_2(H_\alpha) =H_\alpha+1,\qquad
			\sigma_1(H_\beta) =H_\beta, \qquad \sigma_2(H_\beta) =H_\beta-1,\\
			\sigma_1(t_1) =\hat c_1+\hat f_{11}t_1+\hat f_{12}t_2,\qquad \sigma_2(t_1) =t_1,\\
			\sigma_1(t_2) =t_2, \qquad\sigma_2(t_2) =\hat c_2+\hat f_{21}t_1+\hat f_{22}t_2,
		\end{gather*}
	where
	\begin{gather*}
		\hat{c}_1 = -H_\alpha(H_{\beta+\alpha}+1),\qquad
		\hat{c}_2 = -(H_\alpha+2)(H_{\beta+\alpha}+1),\\
		\hat f_{i1} = f_{i1}\frac{(H_\alpha+i)(H_{\beta+\alpha}+1)}{(H_\alpha+2)(H_{\beta+\alpha}+2)},\qquad
		\hat f_{i2} = f_{i2}\frac{(H_\alpha+i)(H_{\beta+\alpha}+1)}{(H_\alpha+1)(H_{\beta+\alpha}+2)},\qquad i=1,2,
	\end{gather*}
	and $f_{ij}=f_{ij}(H_\alpha,H_\beta)$ were defined in \eqref{eq:abcdfij}.
	
\end{Proposition}

\begin{Remark}[commuting automorphisms]
	It is far from obvious that the automorphisms $\sigma_1$ and $\sigma_2$ actually commute; however, we will show indirectly that $\sigma_1$ and $\sigma_2$ commute during an extended proof of Theorem \ref{thm:GWA-realization} found in Appendix~\ref{Appendix:sec:Zero-Commutator}.
\end{Remark}

\begin{Theorem}[$D(\mathfrak{sp}(4))$ is a GWA]\label{thm:GWA-realization}
	With $B$ and $\sigma_{i}$ as in Proposition {\rm\ref{Prop:Normalized-generators-relations}}, there is a $\mathbb{C}$-algebra isomorphism
$\phi\colon B(\sigma,t) \longrightarrow D(\mathfrak{sp}(4))
$
	satisfying
	\[\phi(H_\alpha)=H_\alpha,\qquad
	\phi(H_\beta)=H_\beta,\qquad
	\phi(X_i)=\hat x_i,\qquad
	\phi(Y_i)=\hat \partial_i,\qquad i=1,2,
	\]
	where we use the normalized generators of $D(\mathfrak{sp}(4))$ from Proposition {\rm\ref{Prop:Automorphisms}}.
\end{Theorem}
\begin{proof}
	An extended argument is found in Appendix \ref{Appendix:sec:Zero-Commutator}. The proof relies on the relations of the normalized generators in \eqref{eq:hat-relations} and the finite presentation of $D(\mathfrak{sp}(4))$ found in Theorem~\ref{thm:presentation}.
\end{proof}

\subsection{Rank two skew-affine GWAs}\label{subsec:GWA-skew-affine}
We consider a class of GWAs we call \emph{skew-affine}, restricting to the case of rank two. These GWAs are believed to refine the connection \cite{vandenhomberghNoteMickelssonStep1975} between reduction algebras and GWAs within the study of noncommutative rings.

\begin{Definition}[the skew-affine ansatz for GWAs of rank two]\label{def:SA-GWArank2}
	Suppose $T$ is a finitely generated commutative $\mathbb{C}$-algebra and let $D=T[t_1,t_2]$ be the polynomial algebra over $T$ in two indeterminates $t_i$. Consider the following ansatz for $\mathbb{C}$-algebra automorphisms $\sigma_i$ of $D$
	\begin{gather}\label{eq:SA-GWA}
		\sigma_i(t_i)=c_i+g_{i1}t_1+g_{i2}t_2,\qquad i=1,2,\\
		\sigma_i(t_j)=t_j,\qquad i\neq j,\qquad
		\sigma_i|_T\in\Aut_\mathbb{C}(T),\nonumber
	\end{gather}
	where $c_i$ and $g_{ij}$ are some fixed elements of $T$. A GWA defined by \eqref{eq:SA-GWA} is called a \emph{skew-affine generalized Weyl algebra of rank two}.
\end{Definition}

As a corollary to Theorem \ref{thm:GWA-realization}, $D(\mathfrak{sp}(4))$ is a rank two skew-affine GWA. Its representation theory may be determined by suitable tools of GWAs; moreover, we conjecture that a family of skew-affine GWAs are found as differential reduction algebras of symplectic Lie algebras. We note that the defining relations of skew-affine GWAs are new to GWAs, and it will be useful to address higher-order generalizations in future work.

\appendix

\section[Zero commutator of automorphisms in Theorem 4.6]{Zero commutator of automorphisms in Theorem \ref{thm:GWA-realization}}\label{Appendix:sec:Zero-Commutator}
The computations below show that $\sigma_{1}$ and $\sigma_{2}$ commute.
We omit $\diamond$ for brevity.

\begin{proof}[Extended argument for proof of Theorem \ref{thm:GWA-realization}]
	Let $\phi\colon B\to D(\mathfrak{sp}(4))$ be the $\mathbb{C}$-algebra homomorphism determined by $\phi|_R=\Id_R$ and $\phi(t_i)=\hat\partial_i\hat x_i$ for $i=1,2$. Extend $\phi$ to $B\cup\{X_1,Y_1,X_2,Y_2\}$ by $\phi(X_i)=\hat x_i$, $\phi(Y_i)=\hat\partial_i$. We must show that the GWA relations \eqref{eq:GWA-rels} are satisfied by the images of $\phi$.
	First, we prove that relations \eqref{eq:GWA-rels2} are preserved.
	
	Using $\bar{x}_1 H_\alpha = (H_\alpha-1)\bar{x}_1$,
	\[
	\hat{x}_1\hat{x}_2 = \bar{x}_1(H_\alpha+2)\bar{x}_2 = (H_\alpha+1)\bar{x}_1\bar{x}_2=(H_\alpha+1)\frac{H_\alpha+2}{H_\alpha+1}\bar{x}_2\bar{x}_1=\hat{x}_2\hat{x}_1.
	\]
	
	Using $H_{\beta+\alpha}\bar{\partial}_1=\bar{\partial}_1(H_{\beta+\alpha}-1)$, $\bar{\partial}_2 H_\alpha=(H_\alpha-1)\bar{\partial}_2$, and $\bar{\partial}_2 H_{\beta+\alpha} = (H_{\beta+\alpha}+1)\bar{\partial}_2$, we get
	\begin{align*}
		\hat{\partial}_2\hat{\partial}_1 &= \bar{\partial}_2(H_{\beta+\alpha}+1)\bar{\partial}_1(H_\alpha+1)(H_{\beta+\alpha}+1)=\bar{\partial}_2\bar{\partial}_1(H_\alpha+1)H_{\beta+\alpha}(H_{\beta+\alpha}+1)\\
		&=\bar{\partial}_1\bar{\partial}_2\frac{H_\alpha+2}{H_\alpha+1}(H_\alpha+1)H_{\beta+\alpha}(H_{\beta+\alpha}+1)=
		\bar{\partial}_1\bar{\partial}_2(H_\alpha+2)H_{\beta+\alpha}(H_{\beta+\alpha}+1)\\
		&=\bar{\partial}_1 (H_\alpha+1)(H_{\beta+\alpha}+1)\bar{\partial}_2(H_{\beta+\alpha}+1)=\hat{\partial}_1\hat{\partial}_2.
	\end{align*}
	Next, since $H_{\beta+\alpha}$ commutes with $\bar{x}_1\bar{\partial}_2$, and $H_{\beta+\alpha}\bar{\partial}_2=\bar{\partial}_2(H_{\beta+\alpha}-1)$, we have
	\begin{align*}
		\hat{x}_1\hat{\partial}_2 &= \bar{x}_1\bar{\partial}_2(H_{\beta+\alpha}+1)=
		(H_{\beta+\alpha}+1)\bar{x}_1\bar{\partial}_2=
		(H_{\beta+\alpha}+1)\frac{H_{\beta+\alpha}+2}{H_{\beta+\alpha}+1}\bar{\partial}_2\bar{x}_1\\
		&=(H_{\beta+\alpha}+2)\bar{\partial}_2\bar{x}_1=\bar{\partial}_2(H_{\beta+\alpha}+1)\bar{x}_1=\hat{\partial}_2\hat{x}_1.
	\end{align*}
	From $H_\alpha\bar{x}_2=\bar{x}_2(H_\alpha-1)$, $H_\alpha\bar{\partial}_1=\bar{\partial}_1(H_\alpha-1)$, and $\bar{x}_2 H_{\beta+\alpha} = (H_{\beta+\alpha}-1)\bar{x}_2$, we see
	\begin{align*}
		\hat{x}_2\hat{\partial}_1&=(H_\alpha+2)\bar{x}_2\bar{\partial}_1(H_\alpha+1)(H_{\beta+\alpha}+1)=\bar{x}_2\bar{\partial}_1 H_\alpha (H_\alpha+1)(H_{\beta+\alpha}+1)\\
		&=\bar{\partial}_1\bar{x}_2\frac{H_{\beta+\alpha}+2}{H_{\beta+\alpha}+1}H_\alpha (H_\alpha+1)(H_{\beta+\alpha}+1)=
		\bar{\partial}_1\bar{x}_2 H_\alpha(H_\alpha+1)(H_{\beta+\alpha}+2)\\
		&=\bar{\partial}_1(H_\alpha+1)(H_\alpha+2)(H_{\beta+\alpha}+1)\bar{x}_2=\hat{\partial}_1\hat{x}_2.
	\end{align*}
	We have proved that the four relations in \eqref{eq:GWA-rels2} are preserved by $\phi$.
	
	Next, to show that \eqref{eq:GWA-rels1} holds for all $b\in B$, it suffices to show \eqref{eq:GWA-rels1} holds for $b\in\{H_\alpha, H_\beta, t_1, t_2\}$ since $\sigma_i$ are $\mathbb{C}$-algebra homomorphisms.
	For $b=H_\alpha$ and $b=H_\beta$, \eqref{eq:GWA-rels1} follows directly from \eqref{eq:rel-x1-H}--\eqref{eq:rel-p2-H} and the definition of $\sigma_i$.
	
	For $b=t_i$, the identities are immediately verified once we prove that $\phi(t_i)=\phi(Y_i)\phi(X_i)$ and~${\phi\big(\sigma_i(t_i)\big)=\phi(X_i)\phi(Y_i)}$, because then
	\[\phi(X_i)\phi(t_i)=\phi(X_i)\phi(Y_i)\phi(X_i)=\phi\big(\sigma_i(t_i)\big)\phi(X_i)\]
	and for $j\neq i$
	\[ \phi(X_i)\phi(t_j)=\phi(X_i)\phi(Y_j)\phi(X_j)=\hat{x}_i\hat{\partial}_j\hat{x}_j=\hat{x}_j\hat{\partial}_j\hat{x}_i=\phi\big(\sigma_i(t_j)\big)\phi(X_i);\]
	likewise for $Y_i$.
	
	Actually, that $\phi(t_i)=\phi(Y_i)\phi(X_i)$ is immediate by definition of $\phi(t_i)$. So, what remains is to prove that
	\begin{equation}\label{eq:pf-phi-1}
		\phi\big(\sigma_i(t_i))=\phi(X_i)\phi(Y_i)
	\end{equation}
	for $i=1,2$.
	
	Now the left-hand side of \eqref{eq:pf-phi-1} equals
	\begin{align}
		\phi(\sigma_i(t_i))&=\phi\bigl(c_i+\hat f_{i1}t_1+\hat f_{i2}t_2\bigr)=c_i+\hat f_{i1}\hat{\partial}_1\hat{x}_1+\hat f_{i2}\hat{\partial}_2\hat{x}_2 \nonumber\\
		&=c_i+\hat f_{i1}\cdot(H_\alpha+2)(H_{\beta+\alpha}+2)\bar{\partial}_1\bar{x}_1+\hat f_{i2}\cdot (H_\alpha+1)(H_{\beta+\alpha}+2)\bar{\partial}_2\bar{x}_2\nonumber\\
		&=c_i+(H_\alpha+i)(H_{\beta+\alpha}+1)\bigl(f_{i1}\bar{\partial}_1\bar{x}_1+f_{i2}\bar{\partial}_2\bar{x}_2\bigr);\label{eq:pf-phi-2}
	\end{align}
	indeed, the right-hand side of \eqref{eq:pf-phi-1} equals
	\[\hat{x}_i\hat{\partial}_i=(H_\alpha+i)(H_{\beta+\alpha}+1)\bar{x}_i\bar{\partial}_i,\]
	which equals \eqref{eq:pf-phi-2}, by \eqref{eq:rel-x1p1}.
	
	At this point we can establish that $\sigma_1$ and $\sigma_2$ actually commute.
	We have for any $b\in B$,
	\[0=(\hat{x}_1\hat{x}_2 -\hat{x}_2\hat{x}_1)b = (\sigma_1(\sigma_2(b))-\sigma_2(\sigma_1(b)))\hat{x}_1\hat{x}_2.\]
	By Corollary \ref{cor:Dsp4-structure:domain}, it follows that $\sigma_1$ and $\sigma_2$ commute.
	
	We have shown all GWA relations \eqref{eq:GWA-rels} are preserved by $\phi$. Therefore, there exists a well-defined $\mathbb{C}$-algebra homomorphism $\phi\colon B(\sigma,t)\to D(\mathfrak{sp}(4))$ as in the statement of Theorem \ref{thm:GWA-realization}. Moreover, $\phi$ is a homomorphism of left $R$-modules.
	By \cite{bavula1992generalized}, when $B=R[t_1,t_2,\dots,t_n]$ for some subring $R$, the monomials $Y_1^aY_2^bX_1^cX_2^d$ ($a,b,c,d\in\mathbb{Z}_{\geq 0}$) form a basis for $B(\sigma,t)$ as a left $R$-module. These are mapped under $\phi$ to the monomials $\hat{\partial}_1^a\hat{\partial}_2^b\hat{x}_1^c\hat{x}_2^d$ which form a left $R$-basis for~$D(\mathfrak{sp}(4))$ by Lemma \ref{lemma:finite-generators}.
	Therefore, $\phi$ is bijective.
\end{proof}

\subsection*{Acknowledgements}
The authors thank the referees for their comments inspiring the state of the current paper and suggestions for continued research.
J.T.H.\ is partially supported by the Army Research Office grant W911NF-24-1-0058.

\pdfbookmark[1]{References}{ref}
\LastPageEnding

\end{document}